\documentclass[11pt]{article}

\usepackage[a4paper,margin=2.8cm]{geometry}
\usepackage{amsmath,amssymb,amsthm,mathtools}
\usepackage[T1]{fontenc}
\usepackage{lmodern}
\usepackage[hidelinks]{hyperref}

\newtheorem{theorem}{Theorem}[section]
\newtheorem{proposition}[theorem]{Proposition}
\newtheorem{lemma}[theorem]{Lemma}
\newtheorem{corollary}[theorem]{Corollary}

\theoremstyle{remark}
\newtheorem{remark}[theorem]{Remark}

\newcommand{\C}{\mathbb C}

\title{A Note on Additive Diameter}
\author{Ernesto Ingrosso\\[0.4em]
\small Dipartimento di Matematica e Applicazioni ``Renato Caccioppoli''\\
\small Universit\`a degli Studi di Napoli Federico II, Complesso Universitario Monte S. Angelo\\
\small Via Cintia, Napoli, Italy\\
\small \texttt{ernesto.ingrosso2@unina.it}}
\date{}

\begin{document}
\maketitle

\begin{abstract}
We record two results on additive diameters of group representations. First we give a corrected version of Proposition~5.3 from the first arXiv version of \cite{JezernikSpenko}. Let $0<\varepsilon<1/3$, let $n>9/\varepsilon^2$, and let $U\leq \mathfrak{sl}_n(\C)$ with $\dim U>\varepsilon n^2$. Then

$$
\operatorname{diam}^{\mathrm{SL}_n(\C)}_{+}
\bigl(\mathfrak{sl}_n(\C),U\bigr)
\leq
\frac{32}{\varepsilon}+8.
$$

The proof uses an averaging argument for the action of the symmetric group on the off-diagonal matrix positions. This argument was subsequently developed further in \cite{JezernikSpenko2026}.

We also prove the right-hand inequality in Question~6.7 of \cite{JezernikSpenko}. If $G$ is a complex algebraic group, $V$ is a finite-dimensional $G$-module, $U\leq V$, and $\mathfrak g=\operatorname{Lie}(G)$, then

$$
\operatorname{diam}^{G}_{+}(V,U)
\leq
\operatorname{diam}^{\mathfrak g}_{+}(V,U).
$$

No irreducibility or connectedness assumption is needed.
\end{abstract}

\section{Introduction}

Let $G$ be a group and let $V$ be a finite-dimensional $G$-module. If $U\leq V$, define
$$
\operatorname{diam}^{G}_{+}(V,U)
=
\min\left\{
d:
V=U^{g_1}+\cdots+U^{g_d}
\text{ for some }g_1,\ldots,g_d\in G
\right\},
$$

with value $\infty$ if no such integer exists. Here $U^g$ denotes the image of $U$ under $g$.

If the diameter is finite, then $
\operatorname{diam}^{G}_{+}(V,U)
\geq
\left\lceil\frac{\dim V}{\dim U}\right\rceil.
$ Following \cite{JezernikSpenko}, the diameter is called optimal when equality holds.

Additive diameters of group representations were studied systematically by Jezernik and \v{S}penko in \cite{JezernikSpenko}. Two parts of their work are relevant here.

The first concerns the conjugation action of $\mathrm{SL}_n(\C)$ on $\mathfrak{sl}_n(\C)$. The Borel fixed point theorem permits the reduction of several questions to Borel-stable subspaces. For this action such subspaces are upper right block closed, and the problem becomes one of arranging matrix positions.

In the first arXiv version of \cite{JezernikSpenko}, Proposition~5.3 was intended to prove that a subspace of dimension greater than $\varepsilon n^2$ has additive diameter bounded in terms of $\varepsilon$ alone. The proof of that proposition contained an error.

We give a corrected argument. Suppose that $
0<\varepsilon<\frac13,$ 
$n>\frac9{\varepsilon^2},$ and $\dim U>\varepsilon n^2.$
Then $
\operatorname{diam}^{\mathrm{SL}_n(\C)}_{+}
\bigl(\mathfrak{sl}_n(\C),U\bigr)
\leq
8\left\lceil\frac4\varepsilon\right\rceil
\leq
\frac{32}{\varepsilon}+8.
$

The proof is elementary. If $U$ is upper right block closed, let $S$ be the set of off-diagonal pairs $(i,j)$ such that $E_{ij}\in U$. The hypothesis on $\dim U$ gives a lower bound for $|S|$. Averaging over the action of the symmetric group on the set of ordered pairs $(i,j)$, $i\neq j$, shows that one may choose successive permutations so that each new copy of $S$ covers a fixed proportion of the positions not previously covered. After at most $\lceil 4/\varepsilon\rceil$ steps the sum of the corresponding conjugates has dimension greater than $3n^2/4+n/2$. The diameter bound of \cite{JezernikSpenko} may then be applied, and the Borel reduction removes the stability assumption.

The averaging argument used here was subsequently developed in greater generality in \cite{JezernikSpenko2026}. In that paper the same principle is extended to irreducible representations of compact groups and leads to stronger bounds for a number of representations, including the conjugation representation considered above. We retain the present proof since it gives a direct correction of Proposition~5.3 in the first version of \cite{JezernikSpenko} and records the original symmetric-group argument.

The second question concerns the relation between a group representation and the corresponding Lie algebra representation. Let $G$ be a complex algebraic group, let $V$ be a finite-dimensional $G$-module, and put $\mathfrak g=\operatorname{Lie}(G)$. Differentiation gives a representation $
d\rho.$

For $U\leq V$, the elementary Lie-additive diameter is
$$
\operatorname{diam}^{\mathfrak g}_{+}(V,U)
=
\min\left\{
d+1:
V=
U+d\rho(x_1)U+\cdots+d\rho(x_d)U,
\quad
x_1,\ldots,x_d\in\mathfrak g
\right\}.
$$

Question~6.7 of \cite{JezernikSpenko} asks, under irreducibility and connectedness assumptions, whether
$$
\operatorname{diam}^{\mathfrak g,\mathrm{mon}}_{+}(V,U)
\leq
\operatorname{diam}^{G}_{+}(V,U)
\leq
\operatorname{diam}^{\mathfrak g}_{+}(V,U).
$$

We prove the second inequality without either assumption. Thus, for every finite-dimensional representation of a complex algebraic group and every subspace $U\leq V$,
$$
\operatorname{diam}^{G}_{+}(V,U)
\leq
\operatorname{diam}^{\mathfrak g}_{+}(V,U).
$$

Section~2 contains the definitions and the Borel reduction used below. Section~3 gives the correction to Proposition~5.3 from the first version of \cite{JezernikSpenko}. Section~4 proves the right-hand inequality in Question~6.7.

\section{Preliminaries}

Throughout the paper all vector spaces are defined over $\C$.  If
$W_1,\ldots,W_r$ are subspaces of a vector space $V$, then $ W_1+\cdots+W_r$
denotes their linear span.  We write $E_{ij}$ for the standard matrix unit,
and
$\mathfrak{sl}_n(\C)
 =
 \{A\in M_n(\C):\operatorname{tr}(A)=0\}.$

\subsection{Additive diameters of representations}

Let $G$ be a group and let $\rho:G\longrightarrow \mathrm{GL}(V)$
be a finite-dimensional linear representation.  If $U\leq V$ and $g\in G$,
we put $U^g=\rho(g)U.$
Thus $U^g$ is again a subspace of $V$ and $\dim U^g=\dim U.$

The subspace generated by the $G$-orbit of $U$ is $\langle U^G\rangle
 =
 \sum_{g\in G}U^g.$ This is the smallest $G$-invariant subspace of $V$ containing $U$.

The \emph{$G$-additive diameter of $V$ with respect to $U$} is
\[
 \operatorname{diam}^{G}_{+}(V,U)
 =
 \min\left\{
 m\geq1:
 V=U^{g_1}+\cdots+U^{g_m}
 \text{ for some }g_1,\ldots,g_m\in G
 \right\}.
\]
If no such integer exists, we put $\operatorname{diam}^{G}_{+}(V,U)=\infty.$

The diameter is finite if and only if $U$ generates $V$ as a $G$-module,
that is, $\langle U^G\rangle=V.$
Indeed, if the diameter is finite this is immediate.  Conversely, since
$V$ is finite-dimensional, if $V$ is generated by $\{U^g: g \in G\}$
then finitely many of the summands already span $V$.

If $U\neq0$ and the diameter is finite, then
\[
 \operatorname{diam}^{G}_{+}(V,U)
 \geq
 \left\lceil\frac{\dim V}{\dim U}\right\rceil.
\]
Indeed, a sum of $m$ translates of $U$ has dimension at most
$m\dim U$.  Following \cite{JezernikSpenko}, the diameter is called
\emph{optimal} if equality holds.

When $G$ acts on $M_n(\C)$ or $\mathfrak{sl}_n(\C)$ by conjugation, we use
the convention $A^g=g^{-1}Ag.$
Thus, for a subspace $U$,
\[
 U^g=g^{-1}Ug
 =
 \{g^{-1}ug:u\in U\}.
\]

For the conjugation action on $\mathfrak{sl}_n(\C)$, the actions of
$\mathrm{GL}_n(\C)$ and $\mathrm{SL}_n(\C)$ have the same orbits.
Indeed, if $g\in\mathrm{GL}_n(\C)$, choose $\lambda\in\C^\times$ such that $\det(\lambda g)=1.$
Then $\lambda g\in\mathrm{SL}_n(\C)$ and $(\lambda g)^{-1}A(\lambda g)=g^{-1}Ag.$

\subsection{Lie algebra diameters}

Let now $G$ be a complex algebraic group and put $\mathfrak g=\operatorname{Lie}(G).$
Let $\rho:G\longrightarrow\mathrm{GL}(V)$
be a finite-dimensional representation.  

Differentiation at the identity
gives a Lie algebra representation $d\rho:\mathfrak g\longrightarrow\mathfrak{gl}(V).$ If $x\in\mathfrak g$ and $\gamma(t)$ is a differentiable curve in $G$
such that $\gamma(0)=1$ and $\gamma'(0)=x,$
\hbox{then $d\rho(x)
 =
 \left.\frac{d}{dt}\right|_{t=0}\rho(\gamma(t)).$}

For $x\in\mathfrak g$ and $U\leq V$, we write $d\rho(x)U
 =
 \{d\rho(x)u:u\in U\}.$
Unlike a group translate $U^g$, the space $d\rho(x)U$ need not have the
same dimension as $U$, since $d\rho(x)$ need not be invertible.

The \emph{elementary Lie-additive diameter} of $V$ with respect to $U$ is
\[
 \operatorname{diam}^{\mathfrak g}_{+}(V,U)
 =
 \min\left\{
 d+1:
 V=
 U+d\rho(x_1)U+\cdots+d\rho(x_d)U,
 \quad
 x_1,\ldots,x_d\in\mathfrak g
 \right\}.
\]
If no such decomposition exists, we put $\operatorname{diam}^{\mathfrak g}_{+}(V,U)=\infty.$

We shall also use the monomial Lie-additive diameter appearing in
Question~6.7 of \cite{JezernikSpenko}.  Let $\operatorname{mon}(d\rho(\mathfrak g))$
be the set of all operators $d\rho(x_1)\cdots d\rho(x_r)$, $x_1,\ldots,x_r\in\mathfrak g,$
where $r\geq0$.  When $r=0$ the corresponding monomial is the identity
operator on $V$.

The \emph{monomial Lie-additive diameter} is
\[
 \operatorname{diam}^{\mathfrak g,\mathrm{mon}}_{+}(V,U)
 =
 \min\left\{
 m:
 V=m_1U+\cdots+m_mU,
 \quad
 m_i\in\operatorname{mon}(d\rho(\mathfrak g))
 \right\},
\]
again with value $\infty$ when no such decomposition exists.

The elementary and monomial diameters are different notions.  In the
elementary diameter one may apply only one element of $d\rho(\mathfrak g)$
to each copy of $U$, whereas the monomial diameter allows arbitrary finite
products of such operators \cite[Example~6.3]{JezernikSpenko}.

The passage from the group action to the Lie algebra action is obtained by means of the exponential map
$\exp_G:\mathfrak g\longrightarrow G$.

Let $\rho:G\longrightarrow\operatorname{GL}(V)$ be a finite-dimensional representation, and let $d\rho:\mathfrak g\longrightarrow\mathfrak{gl}(V)$ be the corresponding Lie algebra representation. For $x\in\mathfrak g$ and $t\in\C$ we have
\[
 \rho\bigl(\exp_G(tx)\bigr)
 =
 \exp\bigl(t\,d\rho(x)\bigr).
\]
Here the exponential on the right is the usual exponential of an endomorphism of $V$.

Put $X=d\rho(x)$. Then
\[
 \exp(tX)
 =
 I+tX+\frac{t^2X^2}{2!}+\frac{t^3X^3}{3!}+\cdots,
\]
and hence $\exp(tX)-I=tX+t^2R_X(t),$ where
$R_X(t)=\sum_{m\geq2}\frac{t^{m-2}X^m}{m!}$ is an analytic map from $\C$ to $\operatorname{End}(V)$.

It follows that, for $v\in V$,
$$\bigl(\rho(\exp_G(tx))-I\bigr)v
 =
 t\,d\rho(x)v+t^2z_{x,v}(t),$$
where $z_{x,v}(t)=R_{d\rho(x)}(t)v$ is an analytic map $\C\to V$.

\subsection{Upper right block closed subspaces}

We now recall the terminology used for the conjugation representation of
$\mathrm{SL}_n(\C)$ on $\mathfrak{sl}_n(\C)$.

Let $B$ be the standard Borel subgroup of $\mathrm{SL}_n(\C)$ consisting
of upper triangular matrices.  Its Lie algebra is the standard Borel
subalgebra
\[
 \mathfrak b
 =
 \{A\in\mathfrak{sl}_n(\C):
 A_{ij}=0\text{ whenever }i>j\}.
\]

For $1\leq i,j\leq n$, define
\[
 B_{ij}
 =
\langle E_{k\ell}:k\leq i,\ \ell\geq j \rangle_{\mathbb C}
 \cap\mathfrak{sl}_n(\C).
\]
Thus $B_{ij}$ consists of the traceless matrices supported in the upper
right block determined by the position $(i,j)$.

A subspace $U\leq\mathfrak{sl}_n(\C)$
is called \emph{upper right block closed} if $u\in U, u_{ij}\neq0,
 i\neq j$
implies $B_{ij}\leq U.$
In other words, if an off-diagonal position $(i,j)$ occurs with nonzero
coefficient in some element of $U$, then $U$ contains the whole upper
right block $B_{ij}$.

As recalled in \cite[Section~2.2.2]{JezernikSpenko}, every subspace of
$\mathfrak{sl}_n(\C)$ stable under the standard Borel subgroup is
upper right block closed.

\subsection{Reduction to Borel-stable subspaces}

Let $G$ be a complex linear algebraic group acting on a finite-dimensional
vector space $V$.  For $0<d<\dim V$, let $\operatorname{Gr}(V,d)$
denote the Grassmannian of $d$-dimensional subspaces of $V$.

For $k\geq1$, set
\[
 X_{d,k}
 =
 \left\{
 U\in\operatorname{Gr}(V,d):
 \operatorname{diam}^{G}_{+}(V,U)>k
 \right\}.
\]
The set $X_{d,k}$ is a closed $G$-invariant subvariety of
$\operatorname{Gr}(V,d)$; see
\cite[Section~2.1]{JezernikSpenko}.

If $X_{d,k}$ is nonempty, the Borel fixed point theorem implies that
$X_{d,k}$ contains a subspace fixed by a Borel subgroup of $G$.  We shall
use the following consequence.

\begin{proposition}\label{prop:borelreduction}
Let $G$ be a complex linear algebraic group acting on $V$, and let
$B\leq G$ be a Borel subgroup.  Fix positive integers $d$ and $k$.
Suppose that $\operatorname{diam}^{G}_{+}(V,U)\leq k$
for every $B$-stable subspace $U\in\operatorname{Gr}(V,d).$
Then $\operatorname{diam}^{G}_{+}(V,U)\leq k$
for every $U\in\operatorname{Gr}(V,d).$
\end{proposition}

This is \cite[Proposition~2.3]{JezernikSpenko}.  In the case of the
conjugation action of $\mathrm{SL}_n(\C)$ on $\mathfrak{sl}_n(\C)$,
it reduces questions about arbitrary subspaces to upper right block
closed subspaces.

\subsection{Auxiliary diameter bounds}

We finally record the two estimates from \cite{JezernikSpenko} which will
be used later.

Put $m=\left\lfloor\frac{n+1}{2}\right\rfloor$
and define
\[
 \mathcal B=
 \begin{cases}
 B_{m,m},
   &\text{if }n\text{ is odd},\\[2mm]
 B_{m,m}+B_{m+1,m+1},
   &\text{if }n\text{ is even}.
 \end{cases}
\]
By \cite[Lemma~5.1]{JezernikSpenko}, $\operatorname{diam}^{\mathrm{SL}_n(\C)}_{+}
 \bigl(\mathfrak{sl}_n(\C),\mathcal B\bigr)
 \leq8.$

Moreover, by \cite[Proposition~5.2]{JezernikSpenko}, if $U\leq\mathfrak{sl}_n(\C)$ is upper right block closed and $\dim U>\frac{3n^2}{4}+\frac n2,$
then $ \operatorname{diam}^{\mathrm{SL}_n(\C)}_{+}
 \bigl(\mathfrak{sl}_n(\C),U\bigr)
 \leq8.$

\section{Large subspaces of $\mathfrak{sl}_n(\C)$}

The following elementary averaging observation will be useful.

\begin{lemma}\label{lem:averaging}
Let $\Omega=\{(i,j):1\leq i,j\leq n,\ i\neq j\},$ $M=|\Omega|=n(n-1).$
If $R,S\subseteq\Omega$, then some $\sigma\in S_n$ satisfies
\[
 |R\cap \sigma S|\geq \frac{|R||S|}{M},
\]
where
$\sigma S=\{(\sigma(i),\sigma(j)):(i,j)\in S\}$.
\end{lemma}

\begin{proof}
Fix $(i,j)\in S$ and $(a,b)\in R$.  There are exactly $(n-2)!$
permutations $\sigma\in S_n$ for which
$\sigma(i)=a$ and $\sigma(j)=b$.  Counting the pairs $(\sigma,x)$, $\sigma\in \operatorname{Sym}(n)$, $x\in R\cap \sigma S$, in the two possible ways gives
\[
 \sum_{\sigma\in S_n}|R\cap \sigma S|
 =|R||S|(n-2)!.
\]
Since $n!=M(n-2)!$, the average of the numbers
$|R\cap\sigma S|$ is $|R||S|/M$.  One of them is therefore at least the
average.
\end{proof}

Next result is the correction of Proposition 5.3 of \cite{JezernikSpenko}.

\begin{proposition}\label{prop:corrected53}
Let $0<\varepsilon<1/3$ and $n>9/\varepsilon^2$.  Suppose that
$U\leq\mathfrak{sl}_n(\C)$ is upper right block closed and
$\dim U>\varepsilon n^2$.  Then
\[
 \operatorname{diam}^{\mathrm{SL}_n(\C)}_{+}
 (\mathfrak{sl}_n(\C),U)
 \leq \frac{32}{\varepsilon}+8.
\]
\end{proposition}

\begin{proof}
Put $\Omega=\{(i,j):1\leq i,j\leq n,\ i\neq j\}$, $M=n(n-1)$, and let $S=\{(i,j)\in\Omega:E_{ij}\in U\}.$
Let $D$ denote the traceless diagonal subspace of
$\mathfrak{sl}_n(\C)$.  Since $U$ is upper right block closed, every nonzero
off-diagonal entry occurring in an element of $U$ forces the corresponding
matrix unit to belong to $U$.  It follows that $U=\langle E_{ij}:(i,j)\in S\rangle_{\C}\oplus(U\cap D).$
Consequently $|S|\geq \dim U-(n-1)>\varepsilon n^2-(n-1).$

Now, set $\alpha=|S|/M$.  Then
\[
 \alpha>
 \frac{\varepsilon n^2-(n-1)}{n(n-1)}
 >\varepsilon-\frac1n.
\]
The assumptions on $n$ and $\varepsilon$ give $\frac1n<\frac{\varepsilon^2}{9}<\frac{\varepsilon}{27},$
and hence $\alpha>\frac{26}{27}\varepsilon.$

We now cover most of $\Omega$ by permuted copies of $S$.  Put
$R_0=\Omega$.  Having defined $R_j$, choose $\sigma_{j+1}\in S_n$ so that
$|R_j\cap\sigma_{j+1}S|$ is maximal, and set $R_{j+1}=R_j\setminus\sigma_{j+1}S.$
Lemma~\ref{lem:averaging} yields $|R_j\cap\sigma_{j+1}S|\geq \alpha |R_j|,$
so $|R_{j+1}|\leq(1-\alpha)|R_j|.$

Induction gives $|R_j|\leq M(1-\alpha)^j.$

Take $r=\Bigl\lceil\frac4\varepsilon\Bigr\rceil.$
Since $(1-\alpha)^{-1}>1+\alpha,$ so
\[
 (1-\alpha)^{-r}>(1+\alpha)^r\geq1+r\alpha
 >\frac92.
\]
Thus $|R_r|<\frac29M.$

Let $X=\langle E_{ij}:(i,j)\in S\rangle_{\C}\leq U.$
Every permutation of the standard basis is induced, under conjugation, by
an element of $\mathrm{SL}_n(\C)$: a permutation matrix may be multiplied by
a nonzero scalar without altering its conjugation action.  We may therefore
choose $g_1,\ldots,g_r\in\mathrm{SL}_n(\C)$ such that $X^{g_j}
 =\langle E_{ab}:(a,b)\in\sigma_jS\rangle_{\C}$ for $j \in \{1, \ldots, r\}$.
 
Put $W=X^{g_1}+\cdots+X^{g_r}.$
By the construction of the sets $R_j$, it follows that $W=\langle E_{ij}:(i,j)\in\Omega\setminus R_r\rangle_{\C}.$
Hence, $\dim W=M-|R_r|>\frac79M=\frac79n(n-1).$
Now $n>9/\varepsilon^2>81$, and in particular $n>46$.  Therefore
\[
 \frac79n(n-1)>\frac{3n^2}{4}+\frac n2.
\]
Write $d=\dim W$.

Every Borel-stable subspace of dimension $d$ of
$\mathfrak{sl}_n(\C)$ is upper right block closed.  Since
$d>3n^2/4+n/2$, the argument of
\cite[Lemma~5.1 and Proposition~5.2]{JezernikSpenko} gives
$\mathrm{SL}_n(\C)$-additive diameter at most $8$ for every such
Borel-stable subspace.  The Borel reduction
\cite[Proposition~2.3]{JezernikSpenko} therefore gives $\operatorname{diam}^{\mathrm{SL}_n(\C)}_{+}
 (\mathfrak{sl}_n(\C),W)\leq8.$
Thus there are $h_1,\ldots,h_8\in\mathrm{SL}_n(\C)$ such that $\mathfrak{sl}_n(\C)=W^{h_1}+\cdots+W^{h_8}.$

Since $W\leq U^{g_1}+\cdots+U^{g_r}$, it follows that
\[
 \mathfrak{sl}_n(\C)
=
 \sum_{a=1}^{8}\sum_{b=1}^{r}U^{g_bh_a}.
\]
Consequently, $\operatorname{diam}^{\mathrm{SL}_n(\C)}_{+}
 (\mathfrak{sl}_n(\C),U)
 \leq 8r
 =8\Bigl\lceil\frac4\varepsilon\Bigr\rceil
 \leq\frac{32}{\varepsilon}+8.$
\end{proof}

The Borel reduction immediately removes the stability assumption.

\begin{corollary}\label{cor:allsubspaces}
Let $0<\varepsilon<1/3$ and $n>9/\varepsilon^2$.  If
$U\leq\mathfrak{sl}_n(\C)$ and $\dim U>\varepsilon n^2$, then
\[
 \operatorname{diam}^{\mathrm{SL}_n(\C)}_{+}
 (\mathfrak{sl}_n(\C),U)
 \leq\frac{32}{\varepsilon}+8.
\]
\end{corollary}

\begin{proof}
By \cite[Proposition~2.3]{JezernikSpenko}, it is enough to prove the
bound for Borel-stable subspaces of the same dimension.  Such subspaces are
upper right block closed, so Proposition~\ref{prop:corrected53} applies.
\end{proof}

\section{Group and Lie algebra diameters}

We next compare the group-additive diameter with the elementary Lie-additive
diameter.  The argument is local and does not require irreducibility.

\begin{theorem}\label{thm:rightineq}
Let $G$ be a complex algebraic group, let
$\rho:G\to\operatorname{GL}(V)$ be a finite-dimensional representation,
and let $\mathfrak g=\operatorname{Lie}(G)$. Then, for every subspace
$U\leq V$,
\[
 \operatorname{diam}^{G}_{+}(V,U)
 \leq
 \operatorname{diam}^{\mathfrak g}_{+}(V,U).
\]
\end{theorem}

\begin{proof}
Assume that $\operatorname{diam}^{\mathfrak g}_{+}(V,U)=d+1<\infty.$
Choose $x_1,\ldots,x_d\in\mathfrak g$, and put $X_i=d\rho(x_i),$
so that $V=U+X_1U+\cdots+X_dU.$

We may therefore choose elements $u_1,\ldots,u_r\in U$
and $v_1,\ldots,v_s\in U$, where $i_1,\ldots,i_s\in\{1,\ldots,d\},$
such that $u_1,\ldots,u_r,
 X_{i_1}v_1,\ldots,X_{i_s}v_s$
is a basis of $V$. For $t\in\C$ put $w_j(t)
 =
 \bigl(\operatorname{exp}(tX_{i_j})-I\bigr)v_j.$
 
Since
\[
 \operatorname{exp}(tX_{i_j})-I
 =
 tX_{i_j}+t^2\frac{X_{i_j}^2}{2!}
 +t^3\frac{X_{i_j}^3}{3!}+\cdots,
\]
we have $w_j(t)
 =
 tX_{i_j}v_j+t^2z_j(t)$
for some analytic map $z_j:\mathbb C\to V$.

Fixing a basis of $V$ and taking determinants, multilinearity gives
\[
\begin{aligned}
 &\det\bigl(
 u_1,\ldots,u_r,
 w_1(t),\ldots,w_s(t)
 \bigr)
\\
 &\qquad =
 t^s
 \det\bigl(
 u_1,\ldots,u_r,
 X_{i_1}v_1,\ldots,X_{i_s}v_s
 \bigr)
 +t^{s+1}h(t)
\end{aligned}
\]
for some analytic function $h(t)$.

The first determinant is nonzero. Hence, for some $t\neq0$, $u_1,\ldots,u_r,
 w_1(t),\ldots,w_s(t)$
is again a basis of $V$.

But $u_j\in U$ and $w_j(t)
 =
 \operatorname{exp}(tX_{i_j})v_j-v_j
 \in
 U+\operatorname{exp}(tX_{i_j})U,$ and so $V
 =
 U+\operatorname{exp}(tX_1)U+\cdots+
 \operatorname{exp}(tX_d)U.$

Now set $g_i=\operatorname{exp}(tx_i)\in G.$
Since $\rho(g_i)=\operatorname{exp}(tX_i),$ we obtain $V
 =
 U+\rho(g_1)U+\cdots+\rho(g_d)U.$
Thus $\operatorname{diam}^{G}_{+}(V,U)
 \leq d+1
 =
 \operatorname{diam}^{\mathfrak g}_{+}(V,U).$
\end{proof}
\begin{remark}
Theorem~\ref{thm:rightineq} gives an affirmative answer to the right-hand
inequality in \cite[Question~6.7]{JezernikSpenko}, and in fact proves more
than was asked there: the representation need not be irreducible.
\end{remark}

\section*{Acknowledgement}
The author would like to thank Urban Jezernik for helpful correspondence and for several useful comments and suggestions concerning this note.

The author is member of the non-profit association ``AGTA -- Advances in
Group Theory and Applications'' (www.advgrouptheory.com) and is supported by GNSAGA (INdAM).

\end{document}